\documentclass[12pt]{article}

\usepackage[usenames]{color}
\usepackage{amssymb}

\usepackage{graphicx}
\usepackage{amscd}

\usepackage[colorlinks=true,
linkcolor=webgreen,
filecolor=webbrown,
citecolor=webgreen]{hyperref}

\definecolor{webgreen}{rgb}{0,.5,0}
\definecolor{webbrown}{rgb}{.6,0,0}

\usepackage{color}
\usepackage{fullpage}
\usepackage{float}

\usepackage{graphics,amsmath,amssymb}
\usepackage{amsthm}
\usepackage{amsfonts}
\usepackage{latexsym}
\usepackage{epsf}

\theoremstyle{plain}
\newtheorem{thm}{Theorem}
\newtheorem{lemma}[thm]{Lemma}
\newtheorem{idn}[thm]{Identity}
\newtheorem{cor}[thm]{Corollary}

\theoremstyle{definition}
\newtheorem{example}[thm]{Example}
\newtheorem{remark}[thm]{Remark}
\newtheorem{conj}[thm]{Conjecture}

\newcommand{\Nset}{\mathbb{N}}
\newcommand{\rb}[1]{_{\rm #1}}
\newcommand{\ch}[2]{\begin{bmatrix}#1\\ #2\end{bmatrix}}
\newcommand{\tch}[2]{[\begin{smallmatrix}#1\\ #2\end{smallmatrix}]}
\newcommand{\chb}[2]{\left\langle\begin{matrix}#1\\#2\end{matrix}\right\rangle}
\newcommand{\tchb}[2]{\langle\begin{smallmatrix}#1\\#2\end{smallmatrix}\rangle}
\newcommand{\floor}[1]{\lfloor#1\rfloor}

\newcommand{\seqnum}[1]{\href{https://oeis.org/#1}{\underline{#1}}}

\begin{document}

\begin{center}
\vskip 1cm{\LARGE\bf
On Two Families of Generalizations \\ \vskip .08in
 of Pascal's Triangle}
\vskip 1cm
\large
Michael A. Allen\footnote{Corresponding author.} and Kenneth Edwards\\
Physics Department\\
Faculty of Science\\
Mahidol University\\
Rama 6 Road\\
Bangkok 10400  \\
Thailand \\
\href{mailto:maa5652@gmail.com}{\tt maa5652@gmail.com} \\
\href{mailto:kenneth.edw@mahidol.ac.th}{\tt kenneth.edw@mahidol.ac.th} \\
\end{center}

\vskip .2 in

\begin{abstract}
We consider two families of Pascal-like triangles that have all ones
on the left side and ones separated by $m-1$ zeros on the right
side. The $m=1$ cases are Pascal's triangle and the two families also
coincide when $m=2$. Members of the first family obey Pascal's
recurrence everywhere inside the triangle. We show that the $m$-th
triangle can also be obtained by reversing the elements up to and
including the main diagonal in each row of the $(1/(1-x^m),x/(1-x))$
Riordan array.  Properties of this family of triangles can be
obtained quickly as a result.  The $(n,k)$-th entry in the $m$-th
member of the second family of triangles is the number of tilings of
an $(n+k)\times1$ board that use $k$ $(1,m-1)$-fences and $n-k$ unit
squares.  A $(1,g)$-fence is composed of two unit square sub-tiles
separated by a gap of width $g$. We show that the entries in the
antidiagonals of these triangles are coefficients of products of
powers of two consecutive Fibonacci polynomials and give a bijective
proof that these coefficients give the number of $k$-subsets of
$\{1,2,\ldots,n-m\}$ such that no two elements of a subset differ by
$m$.  Other properties of the second family of triangles are also
obtained via a combinatorial approach.  Finally, we give necessary and
sufficient conditions for any Pascal-like triangle (or its
row-reversed version) derived from tiling $(n\times1)$-boards to be a Riordan
array.
\end{abstract}

\section{Introduction}

Pascal's triangle, whose entries $\tbinom{n}{k}$ satisfy Pascal's
recurrence,
\[
\binom{n}{k}=\binom{n-1}{k}+\binom{n-1}{k-1},
\]
is also the $(1-x,x/(1-x))$ Riordan array. A \textit{$(p(x),q(x))$ Riordan
array}, where $p(x)=p_0+p_1x+p_2x^2+\cdots$ and
$q(x)=q_1x+q_2x^2+\cdots$, is an infinite lower triangular matrix
whose $(n,k)$-th entry (where $n\ge0,k\ge0$) is denoted and defined by
$(p(x),q(x))_{n,k}=[x^n]p(x)(q(x))^k$, where the coefficient operator
$[x^n]$ gives the coefficient of $x^n$ in the series expansion of the
term it precedes \cite{SGWW91,Bar=16}. We define the
\textit{row-reversed $(p,q)$ Riordan array} as the lower triangular
matrix obtained by reversing the elements of each row up to and
including the main diagonal of the $(p,q)$ Riordan array, i.e., the
$(n,k)$-th element of the row-reversed $(p,q)$ Riordan array is
$(p(x),q(x))_{n,n-k}$.  Notice that whereas the definition of a
$(p,q)$ Riordan array implies that the 0th column of the array gives
the coefficients of the generating function of $p(x)$, it is the main
diagonal of the row-reversed Riordan array where these coefficients
appear.  Owing to the symmetry of its rows, Pascal's triangle is also
the row-reversed $(1-x,x/(1-x))$ Riordan array. Note, however, that in
general, a row-reversed Riordan array is not a Riordan array.

Pascal's triangle also has tiling interpretations. The $(n,k)$-th
entry (when written as a lower triangular matrix with the first 1
taken as the $(0,0)$-th entry) is the number of square-and-domino
tilings of $N$-boards for any $N$ (an $N$-board is a linear array of
$N$ unit square cells) that use $n$ tiles in total of which $k$ are
dominoes (and therefore $n-k$ are squares). This is easily seen since
there are $\tbinom{n}{k}$ ways to choose which $k$ of the $n$ tiles
are dominoes. As there are $2^n$ different possible square-and-domino
$n$-tile tilings, one immediately has a combinatorial proof that
$\sum_{k=0}^n\tbinom{n}{k}=2^n$ \cite{BQ=03}.  Also, the $k$-th entry
in the $n$-th antidiagonal (i.e., $\tbinom{n-k}{k}$) is the number of
tilings of an $n$-board that use $k$ dominoes and $n-2k$ squares. The
number of ways to tile an $n$-board using squares and dominoes is the
Fibonacci number $f_n$ given by $f_n=f_{n-1}+f_{n-2}+\delta_{n,0}$,
$f_{n<0}=0$ (\seqnum{A000045}(n+1) in the OEIS \cite{Slo-OEIS}), where
$\delta_{i,j}$ is 1 if $i=j$ and zero otherwise. This is one way to
show that the sum of elements of the $n$-th antidiagonal of Pascal's
triangle is $f_n$ \cite{BQ=03}.

A \textit{$(w,g)$-fence} is a tile composed of two sub-tiles (called
\textit{posts}) of dimensions $w\times1$ which are separated by a gap
of width $g$.  Fences have been used to give tiling interpretations of
various sequences \cite{Edw08,EA19,EA20a,EA21}. They have also
been employed in combinatorial proofs relating to strongly restricted
permutations \cite{EA15} and a probability problem \cite{DEM21}.

\begin{figure}
\begin{tabular}{c|*{13}{p{1.6em}}}
$n$ $\backslash$ $k$&\mbox{}\hfill0&\mbox{}\hfill1&\mbox{}\hfill2&\mbox{}\hfill3&\mbox{}\hfill4&\mbox{}\hfill5&\mbox{}\hfill6&\mbox{}\hfill7&\mbox{}\hfill8&\mbox{}\hfill9&\mbox{}\hfill10&\mbox{}\hfill11&\mbox{}\hfill12\\\hline
 0~~&\mbox{}\hspace*{\fill}\textbf{  1}\\
 1~~&\mbox{}\hspace*{\fill}\textbf{  1}&\mbox{}\hspace*{\fill}\textbf{  0}\\
 2~~&\mbox{}\hspace*{\fill}\textbf{  1}&\mbox{}\hspace*{\fill}\textbf{  1}&\mbox{}\hspace*{\fill}\textbf{  1}\\
 3~~&\mbox{}\hspace*{\fill}\textbf{  1}&\mbox{}\hspace*{\fill}\textbf{  2}&\mbox{}\hspace*{\fill}\textbf{  2}&\mbox{}\hspace*{\fill}\textbf{  0}\\
 4~~&\mbox{}\hspace*{\fill}\textbf{  1}&\mbox{}\hspace*{\fill}\textbf{  3}&\mbox{}\hspace*{\fill}\textbf{  4}&\mbox{}\hspace*{\fill}\textbf{  2}&\mbox{}\hspace*{\fill}\textbf{  1}\\
 5~~&\mbox{}\hspace*{\fill}\textbf{  1}&\mbox{}\hspace*{\fill}\textbf{  4}&\mbox{}\hfill  7 &\mbox{}\hspace*{\fill}\textbf{  6}&\mbox{}\hspace*{\fill}\textbf{  3}&\mbox{}\hspace*{\fill}\textbf{  0}\\
 6~~&\mbox{}\hspace*{\fill}\textbf{  1}&\mbox{}\hspace*{\fill}\textbf{  5}&\mbox{}\hfill 11 &\mbox{}\hfill 13 &\mbox{}\hspace*{\fill}\textbf{  9}&\mbox{}\hspace*{\fill}\textbf{  3}&\mbox{}\hspace*{\fill}\textbf{  1}\\
 7~~&\mbox{}\hspace*{\fill}\textbf{  1}&\mbox{}\hspace*{\fill}\textbf{  6}&\mbox{}\hfill 16 &\mbox{}\hfill 24 &\mbox{}\hfill 22 &\mbox{}\hspace*{\fill}\textbf{ 12}&\mbox{}\hspace*{\fill}\textbf{  4}&\mbox{}\hspace*{\fill}\textbf{  0}\\
 8~~&\mbox{}\hspace*{\fill}\textbf{  1}&\mbox{}\hspace*{\fill}\textbf{  7}&\mbox{}\hfill 22 &\mbox{}\hfill 40 &\mbox{}\hfill 46 &\mbox{}\hfill 34 &\mbox{}\hspace*{\fill}\textbf{ 16}&\mbox{}\hspace*{\fill}\textbf{  4}&\mbox{}\hspace*{\fill}\textbf{  1}\\
 9~~&\mbox{}\hspace*{\fill}\textbf{  1}&\mbox{}\hspace*{\fill}\textbf{  8}&\mbox{}\hfill 29 &\mbox{}\hfill 62 &\mbox{}\hfill 86 &\mbox{}\hfill 80 &\mbox{}\hfill 50 &\mbox{}\hspace*{\fill}\textbf{ 20}&\mbox{}\hspace*{\fill}\textbf{  5}&\mbox{}\hspace*{\fill}\textbf{  0}\\
10~~&\mbox{}\hspace*{\fill}\textbf{  1}&\mbox{}\hspace*{\fill}\textbf{  9}&\mbox{}\hfill 37 &\mbox{}\hfill 91 &\mbox{}\hfill148 &\mbox{}\hfill166 &\mbox{}\hfill130 &\mbox{}\hfill 70 &\mbox{}\hspace*{\fill}\textbf{ 25}&\mbox{}\hspace*{\fill}\textbf{  5}&\mbox{}\hspace*{\fill}\textbf{  1}\\
11~~&\mbox{}\hspace*{\fill}\textbf{  1}&\mbox{}\hspace*{\fill}\textbf{ 10}&\mbox{}\hfill 46 &\mbox{}\hfill128 &\mbox{}\hfill239 &\mbox{}\hfill314 &\mbox{}\hfill296 &\mbox{}\hfill200 &\mbox{}\hfill 95 &\mbox{}\hspace*{\fill}\textbf{ 30}&\mbox{}\hspace*{\fill}\textbf{  6}&\mbox{}\hspace*{\fill}\textbf{  0}\\
12~~&\mbox{}\hspace*{\fill}\textbf{  1}&\mbox{}\hspace*{\fill}\textbf{ 11}&\mbox{}\hfill 56 &\mbox{}\hfill174 &\mbox{}\hfill367 &\mbox{}\hfill553 &\mbox{}\hfill610 &\mbox{}\hfill496 &\mbox{}\hfill295 &\mbox{}\hfill125 &\mbox{}\hspace*{\fill}\textbf{ 36}&\mbox{}\hspace*{\fill}\textbf{  6}&\mbox{}\hspace*{\fill}\textbf{  1}\\
\end{tabular}
\caption{A Pascal-like triangle (\seqnum{A059259}) whose $(n,k)$-th entry,
  $\protect\binom{n}{k}_2=\protect\tchb{n}{k}_2$, is the $(n,k)$-th
  element of the row-reversed $(1/(1-x^2),x/(1-x))$ Riordan array and
  also the number of $n$-tile tilings using $k$ $(1,1)$-fences (and
  $n-k$ squares). Entries in bold font (and those in bold font in
  Figs. \ref{f:m=3t} and \ref{f:m=4t}) are covered by identities in \S\ref{s:idn}. }
\label{f:m=2}
\end{figure}

A $(1,m-1)$-fence for $m=2,3,\dots$ can be regarded as a generalization
of a domino since the $m=1$ case is effectively a domino. If one
creates a Pascal-like triangle by specifying that the $(n,k)$-th entry
is the number of $n$-tile tilings of a board that use $k$
$(1,1)$-fences and $n-k$ squares one arrives at the triangle
\seqnum{A059259} whose entries satisfy Pascal's recurrence
\cite{EA21} (Fig.~\ref{f:m=2}). The triangle is also the row-reversed
$(1/(1-x^2),x/(1-x))$ Riordan array. One naturally asks whether the
triangle generated by tiling boards with $(1,m-1)$-fences and squares
for a given fixed $m$ is the row-reversed $(1/(1-x^m),1/(1-x))$
array. The answer for $m>2$, as we will show here, is no, and so we obtain
two separate families of triangles which only coincide for the
$m=1,2$ cases. However, one feature that the families have in common
is their sides: the left sides are all ones and the right side of the
$m$-th member of each family is the repetition of 1 followed by $m-1$
zeros.

Our main concern here are triangles generated from tiling with squares
and $(1,m-1)$-fences. However, for completeness, in \S\ref{s:fam1} we
look at triangles with the sides specified above that obey Pascal's
recurrence everywhere in the interior. These turn out to be the same
as the row-reversed $(1/(1-x^m),x/(1-x))$ Riordan arrays. The start of
the section also serves as an introduction to \S\ref{s:rio} where we
discuss in general which tiling-derived triangles can be Riordan
arrays or their row-reversed versions, and the remainder of
\S\ref{s:fam1} gives us the opportunity to illustrate how to obtain
generating functions for sums of antidiagonals and bivariate
generating functions for row-reversed Riordan arrays which does not
seem to have been addressed elsewhere in the literature. In
\S\ref{s:fam2} we introduce the tiling-derived family of triangles
along with a closely related family which is helpful in proving some
of the properties of the triangles; the rows of the latter family are
the antidiagonals of the former. The tiling-derived triangles are
shown to be related to Fibonacci polynomials and restricted
combinations in \S\ref{s:poly} and \S\ref{s:comb}, respectively. In
\S\ref{s:meta} we give general properties of tiling with squares and
$(1,m-1)$-fences which are used in the proofs of identities in
\S\ref{s:idn}.

\section{Triangles obtained from Pascal's recurrence}\label{s:fam1}

We denote the $(n,k)$-th entry of the $m$th member of our first family
of generalizations of Pascal's triangle by $\tbinom{n}{k}_m$. By
definition, we require that for
$n\ge0$, $\tbinom{n}{0}_m=1$ and $\tbinom{n}{n}_m=\delta_{n\bmod
  m,0}$, and that inside the triangle we have
\begin{equation}\label{e:PR} 
\binom{n}{k}_m=\binom{n-1}{k}_m+\binom{n-1}{k-1}_m, \quad 0<k<n.
\end{equation} 
For $m=1,\ldots,5$ the triangles are Pascal's triangle (\seqnum{A007318}),
\seqnum{A059259}, \seqnum{A118923}, \seqnum{A349839}, and
\seqnum{A349841}, respectively. The $m=2,3,4$ cases are displayed in
Figs.~\ref{f:m=2}, \ref{f:m=3}, and \ref{f:m=4}, respectively.

\begin{figure}
\begin{tabular}{c|*{14}{p{1.5em}}}
$n$ $\backslash$ $k$&\mbox{}\hfill0&\mbox{}\hfill1&\mbox{}\hfill2&\mbox{}\hfill3&\mbox{}\hfill4&\mbox{}\hfill5&\mbox{}\hfill6&\mbox{}\hfill7&\mbox{}\hfill8&\mbox{}\hfill9&\mbox{}\hfill10&\mbox{}\hfill11&\mbox{}\hfill12\\\hline
 0~~&\mbox{}\hfill  1 \\
 1~~&\mbox{}\hfill  1 &\mbox{}\hfill  0 \\
 2~~&\mbox{}\hfill  1 &\mbox{}\hfill  1 &\mbox{}\hfill  0 \\
 3~~&\mbox{}\hfill  1 &\mbox{}\hfill  2 &\mbox{}\hfill  1 &\mbox{}\hfill  1 \\
 4~~&\mbox{}\hfill  1 &\mbox{}\hfill  3 &\mbox{}\hfill  3 &\mbox{}\hfill  2 &\mbox{}\hfill  0 \\
 5~~&\mbox{}\hfill  1 &\mbox{}\hfill  4 &\mbox{}\hfill  6 &\mbox{}\hfill  5 &\mbox{}\hfill  2 &\mbox{}\hfill  0 \\
 6~~&\mbox{}\hfill  1 &\mbox{}\hfill  5 &\mbox{}\hfill 10 &\mbox{}\hfill 11 &\mbox{}\hfill  7 &\mbox{}\hfill  2 &\mbox{}\hfill  1 \\
 7~~&\mbox{}\hfill  1 &\mbox{}\hfill  6 &\mbox{}\hfill 15 &\mbox{}\hfill 21 &\mbox{}\hfill 18 &\mbox{}\hfill  9 &\mbox{}\hfill  3 &\mbox{}\hfill  0 \\
 8~~&\mbox{}\hfill  1 &\mbox{}\hfill  7 &\mbox{}\hfill 21 &\mbox{}\hfill 36 &\mbox{}\hfill 39 &\mbox{}\hfill 27 &\mbox{}\hfill 12 &\mbox{}\hfill  3 &\mbox{}\hfill  0 \\
 9~~&\mbox{}\hfill  1 &\mbox{}\hfill  8 &\mbox{}\hfill 28 &\mbox{}\hfill 57 &\mbox{}\hfill 75 &\mbox{}\hfill 66 &\mbox{}\hfill 39 &\mbox{}\hfill  15 &\mbox{}\hfill  3 &\mbox{}\hfill  1 \\
10~~&\mbox{}\hfill  1 &\mbox{}\hfill  9 &\mbox{}\hfill 36 &\mbox{}\hfill 85 &\mbox{}\hfill132 &\mbox{}\hfill141 &\mbox{}\hfill105 &\mbox{}\hfill 54 &\mbox{}\hfill 18 &\mbox{}\hfill  4 &\mbox{}\hfill  0 \\
11~~&\mbox{}\hfill  1 &\mbox{}\hfill  10 &\mbox{}\hfill 45 &\mbox{}\hfill 121 &\mbox{}\hfill217 &\mbox{}\hfill273 &\mbox{}\hfill246 &\mbox{}\hfill159 &\mbox{}\hfill 72 &\mbox{}\hfill 22 &\mbox{}\hfill  4 &\mbox{}\hfill  0 \\
12~~&\mbox{}\hfill  1 &\mbox{}\hfill  11 &\mbox{}\hfill 55 &\mbox{}\hfill166 &\mbox{}\hfill338 &\mbox{}\hfill490 &\mbox{}\hfill519 &\mbox{}\hfill405 &\mbox{}\hfill231 &\mbox{}\hfill 94 &\mbox{}\hfill 26 &\mbox{}\hfill  4 &\mbox{}\hfill  1 \\
\end{tabular}
\caption{A Pascal-like triangle (\seqnum{A118923}) with entries
  $\protect\binom{n}{k}_3$.}
\label{f:m=3}
\end{figure}

\begin{figure}
\begin{tabular}{c|*{14}{p{1.5em}}}
$n$ $\backslash$ $k$&\mbox{}\hfill0&\mbox{}\hfill1&\mbox{}\hfill2&\mbox{}\hfill3&\mbox{}\hfill4&\mbox{}\hfill5&\mbox{}\hfill6&\mbox{}\hfill7&\mbox{}\hfill8&\mbox{}\hfill9&\mbox{}\hfill10&\mbox{}\hfill11&\mbox{}\hfill12\\\hline
 0~~&\mbox{}\hfill  1 \\
 1~~&\mbox{}\hfill  1 &\mbox{}\hfill  0 \\
 2~~&\mbox{}\hfill  1 &\mbox{}\hfill  1 &\mbox{}\hfill  0 \\
 3~~&\mbox{}\hfill  1 &\mbox{}\hfill  2 &\mbox{}\hfill  1 &\mbox{}\hfill  0 \\
 4~~&\mbox{}\hfill  1 &\mbox{}\hfill  3 &\mbox{}\hfill  3 &\mbox{}\hfill  1 &\mbox{}\hfill  1 \\
 5~~&\mbox{}\hfill  1 &\mbox{}\hfill  4 &\mbox{}\hfill  6 &\mbox{}\hfill  4 &\mbox{}\hfill  2 &\mbox{}\hfill  0 \\
 6~~&\mbox{}\hfill  1 &\mbox{}\hfill  5 &\mbox{}\hfill 10 &\mbox{}\hfill 10 &\mbox{}\hfill  6 &\mbox{}\hfill  2 &\mbox{}\hfill  0 \\
 7~~&\mbox{}\hfill  1 &\mbox{}\hfill  6 &\mbox{}\hfill 15 &\mbox{}\hfill 20 &\mbox{}\hfill 16 &\mbox{}\hfill  8 &\mbox{}\hfill  2 &\mbox{}\hfill  0 \\
 8~~&\mbox{}\hfill  1 &\mbox{}\hfill  7 &\mbox{}\hfill 21 &\mbox{}\hfill 35 &\mbox{}\hfill 36 &\mbox{}\hfill 24 &\mbox{}\hfill 10 &\mbox{}\hfill  2 &\mbox{}\hfill  1 \\
 9~~&\mbox{}\hfill  1 &\mbox{}\hfill  8 &\mbox{}\hfill 28 &\mbox{}\hfill 56 &\mbox{}\hfill 71 &\mbox{}\hfill 60 &\mbox{}\hfill 34 &\mbox{}\hfill 12 &\mbox{}\hfill  3 &\mbox{}\hfill  0 \\
10~~&\mbox{}\hfill  1 &\mbox{}\hfill  9 &\mbox{}\hfill 36 &\mbox{}\hfill 84 &\mbox{}\hfill127 &\mbox{}\hfill131 &\mbox{}\hfill94 &\mbox{}\hfill 46 &\mbox{}\hfill  15 &\mbox{}\hfill  3 &\mbox{}\hfill  0 \\
11~~&\mbox{}\hfill  1 &\mbox{}\hfill  10 &\mbox{}\hfill 45 &\mbox{}\hfill120 &\mbox{}\hfill211 &\mbox{}\hfill258 &\mbox{}\hfill225 &\mbox{}\hfill 140 &\mbox{}\hfill 61 &\mbox{}\hfill  18 &\mbox{}\hfill  3 &\mbox{}\hfill  0 \\
12~~&\mbox{}\hfill  1 &\mbox{}\hfill  11 &\mbox{}\hfill 55 &\mbox{}\hfill165 &\mbox{}\hfill331 &\mbox{}\hfill469 &\mbox{}\hfill483 &\mbox{}\hfill365 &\mbox{}\hfill 201 &\mbox{}\hfill 79 &\mbox{}\hfill  21 &\mbox{}\hfill  3 &\mbox{}\hfill  1 \\
\end{tabular}
\caption{A Pascal-like triangle (\seqnum{A349839}) with entries
  $\protect\binom{n}{k}_4$.}
\label{f:m=4}
\end{figure}

In the following theorem we link this family of triangles to Riordan
arrays. To do so we need the result that if 
\begin{equation}\label{e:Aseq} 
(p,q)_{n,k}=A_0(p,q)_{n-1,k-1}+A_1(p,q)_{n-1,k}+\cdots+A_j(p,q)_{n-1,k-1+j}+\cdots,
\end{equation} 
where the so-called $A$-sequence $\{A_j\}_{j\ge0}$ are the coefficients of the
generating function $A(x)$, then $q(x)=xA(q(x))$ \cite{Rog78,Spr94}.

\begin{thm}\label{T:R3A}
The triangle whose entries are $\tbinom{n}{k}_m$ as defined above is
the row-reversed $(1/(1-x^m),x/(1-x))$ Riordan array.
\end{thm}
\begin{proof}
Row-reversing the elements of the triangle gives a lower triangular
matrix whose 0th column is the repetition of 1 followed by $m-1$
zeros, i.e., the coefficients of the series expansion of
$1/(1-x^m)$. Hence if the triangle is a row-reversed $(p,q)$ Riordan
array then $p=1/(1-x^m)$. Rewriting the recursion relation
\eqref{e:PR} in terms of the elements $(p,q)_{n,k}$ of the
row-reversed triangle and then replacing $n-k$ by $k$ gives
\[
(p,q)_{n,k}=(p,q)_{n-1,k-1}+(p,q)_{n-1,k}.
\]
Hence $A(x)=1+x$ and so $q=xA(q)=x(1+q)$ from which $q=x/(1-x)$. 
\end{proof}

A number of properties of the triangles can then be obtained easily.
The first of these follows immediately from the definition of a Riordan array.
\begin{cor}
\begin{equation}\label{e:tri1entry} 
\binom{n}{k}_m=[x^n]\frac1{1-x^m}\left(\frac{x}{1-x}\right)^{n-k}.
\end{equation} 
\end{cor}

The next result is 
a more explicit expression for the general term in the triangles.
\begin{cor}
$\tbinom{n}{n}_m=\delta_{n \bmod m,0}$ and, for $0\le k<n$, 
\begin{equation}\label{e:genterm}  
\binom{n}{k}_m=
\sum_{j=0}^{\floor{k/m}}\binom{n-mj-1}{n-k-1}.
\end{equation} 
\end{cor}
\begin{proof}
The expression for $\tbinom{n}{n}_m$ follows immediately from the
definition of the array. For the other values, 
we start with the definition of the $(n,k)$-th element of the Riordan
array which for $k>0$ gives
\[
\left(\frac1{1-x^m},\frac{x}{1-x}\right)_{n,k}=
[x^n]\frac1{1-x^m}\frac{x^k}{(1-x)^k}\\
=[x^{n-k}]\sum_{j=0}^\infty x^{mj}\sum_{r=0}^\infty\binom{k+r-1}{r}x^r.
\]
To obtain the coefficient of $x^{n-k}$ we only require the $r=n-k-mj$
term in the sum over $r$, and as $r$ cannot be negative, $j$ cannot
exceed $(n-k)/m$. This leaves
\[
\left(\frac1{1-x^m},\frac{x}{1-x}\right)_{n,k}=
\sum_{j=0}^{\floor{(n-k)/m}}\binom{n-mj-1}{n-k-mj}
=\sum_{j=0}^{\floor{(n-k)/m}}\binom{n-mj-1}{k-1}.
\] 
The result then follows from Theorem~\ref{T:R3A}.
\end{proof}

The $k$-th column of a $(p,q)$ Riordan array, whose generating function
is $pq^k$, becomes the $k$-th subdiagonal (counting the main
diagonal as the 0th subdiagonal) of the row-reversed $(p,q)$
Riordan array after removing the initial $k$ zeros. Hence the
corresponding generating function is $p(x)(q(x)/x)^k$ and we have the
following result.

\begin{cor}
The generating function for the $k$-th subdiagonal is $1/((1-x^m)(1-x)^k)$.
\end{cor}

The next result follows from the fact that
the generating function for the sums of the rows of a $(p,q)$
Riordan array (and therefore also the corresponding row-reversed
array) is $p/(1-q)$ \cite{Spr94}. 
\begin{cor}
The generating function $g\rb{r}(x)$ for the row sums of the $m$-th triangle
is given by
\begin{equation}\label{e:rowsumgf} 
g\rb{r}(x)=\frac{1-x}{(1-x^m)(1-2x)}=\frac1{(1+x+\cdots+x^{m-1})(1-2x)}.
\end{equation}
\end{cor} 
From \eqref{e:rowsumgf}, the recursion relation giving $r_n$, the sum
of the $n$-th row, can be expressed as
$r_n=2r_{n-1}+r_{n-m}-2r_{n-m-1}+\delta_{n,0}-\delta_{n,1}$ or
$r_n=r_{n-1}+\cdots+r_{n-m+1}+2r_{n-m}+\delta_{n,0}$. For
$m=1,\dots,5$ these correspond to the sequences \seqnum{A000079},
\seqnum{A001045}$(n+1)$, \seqnum{A077947}, \seqnum{A115451}, and
\seqnum{A349842}, respectively.

The sum of the $n$-th antidiagonal of a row-reversed $(p,q)$ Riordan
array (counting the $(0,0)$-th element as the 0th antidiagonal)
can be obtained from the generating function
$p(x^2)/(1-q(x^2)/x)$. This leads to the following result for our
triangles.
\begin{cor}
For the $m$-th triangle, the generating function $g\rb{a}$ for the
antidiagonal sums is given by
\begin{equation}\label{e:diagsumgf} 
g\rb{a}(x)=\frac{1-x^2}{(1-x^{2m})(1-x-x^2)}
=\frac1{(1+x^2+\cdots+x^{2(m-1)})(1-x-x^2)}.
\end{equation}
\end{cor} 
Thus the recursion relation for $a_n$, the sum of the $n$th
antidiagonal, can be written as 
$a_n=a_{n-1}+a_{n-2}+a_{n-2m}-a_{n-2m-1}-a_{n-2m-2}+\delta_{n,0}-\delta_{n,2}$
or 
$a_n=a_{n-1}+a_{n-3}+\cdots+a_{n-2m+1}+a_{n-2m}+\delta_{n,0}$.
For
$m=1,\dots,5$ these correspond to the sequences \seqnum{A000045}$(n+1)$,
\seqnum{A006498}, \seqnum{A079962}, \seqnum{A349840}, and
\seqnum{A349843}, respectively. 

The bivariate generating function for the row-reversed $(p,q)$ Riordan
array is $p(xy)/(1-q(xy)/y)$. 
\begin{cor}
The bivariate generating function
for the $m$-th triangle is given by
\begin{equation}\label{e:bgf}
g_m(x,y)=\frac{1-xy}{(1-(xy)^m)(1-x-xy)}
=\frac1{(1+xy+\cdots+(xy)^{m-1})(1-x-xy)}.
\end{equation} 
\end{cor}
Note that it is a generating function in the sense that
$\tbinom{n}{k}_m=[x^ny^k]g_m(x,y)$.

\section{Triangles derived from tiling}\label{s:fam2}

For $m=1,2,\dots$, let $\tchb{n}{k}_m$ denote the number of
$n$-tile tilings that use $k$ $(1,m-1)$-fences (and $n-k$ squares).  The
$m=2,3,4$ cases are shown in Figs.~\ref{f:m=2}, \ref{f:m=3t}, and
\ref{f:m=4t}. The cases $m=1,\dots,5$ are sequences \seqnum{A007318},
\seqnum{A059259}, \seqnum{A350110}, \seqnum{A350111}, and
\seqnum{A350112}, respectively.

\begin{figure}
\begin{tabular}{c|*{14}{p{1.6em}}}
$n$ $\backslash$ $k$&\mbox{}\hfill0&\mbox{}\hfill1&\mbox{}\hfill2&\mbox{}\hfill3&\mbox{}\hfill4&\mbox{}\hfill5&\mbox{}\hfill6&\mbox{}\hfill7&\mbox{}\hfill8&\mbox{}\hfill9&\mbox{}\hfill10&\mbox{}\hfill11&\mbox{}\hfill12&\mbox{}\hfill13\\\hline
 0~~&\mbox{}\hspace*{\fill}\textbf{  1}\\
 1~~&\mbox{}\hspace*{\fill}\textbf{  1}&\mbox{}\hspace*{\fill}\textbf{  0}\\
 2~~&\mbox{}\hspace*{\fill}\textbf{  1}&\mbox{}\hspace*{\fill}\textbf{  0}&\mbox{}\hspace*{\fill}\textbf{  0}\\
 3~~&\mbox{}\hspace*{\fill}\textbf{  1}&\mbox{}\hspace*{\fill}\textbf{  1}&\mbox{}\hspace*{\fill}\textbf{  1}&\mbox{}\hspace*{\fill}\textbf{  1}\\
 4~~&\mbox{}\hspace*{\fill}\textbf{  1}&\mbox{}\hspace*{\fill}\textbf{  2}&\mbox{}\hspace*{\fill}\textbf{  3}&\mbox{}\hspace*{\fill}\textbf{  2}&\mbox{}\hspace*{\fill}\textbf{  0}\\
 5~~&\mbox{}\hspace*{\fill}\textbf{  1}&\mbox{}\hspace*{\fill}\textbf{  3}&\mbox{}\hfill  5 &\mbox{}\hspace*{\fill}\textbf{  4}&\mbox{}\hspace*{\fill}\textbf{  0}&\mbox{}\hspace*{\fill}\textbf{  0}\\
 6~~&\mbox{}\hspace*{\fill}\textbf{  1}&\mbox{}\hspace*{\fill}\textbf{  4}&\mbox{}\hfill  8 &\mbox{}\hspace*{\fill}\textbf{  8}&\mbox{}\hspace*{\fill}\textbf{  4}&\mbox{}\hspace*{\fill}\textbf{  2}&\mbox{}\hspace*{\fill}\textbf{  1}\\
 7~~&\mbox{}\hspace*{\fill}\textbf{  1}&\mbox{}\hspace*{\fill}\textbf{  5}&\mbox{}\hfill 12 &\mbox{}\hfill 16 &\mbox{}\hfill 13 &\mbox{}\hspace*{\fill}\textbf{  9}&\mbox{}\hspace*{\fill}\textbf{  3}&\mbox{}\hspace*{\fill}\textbf{  0}\\
 8~~&\mbox{}\hspace*{\fill}\textbf{  1}&\mbox{}\hspace*{\fill}\textbf{  6}&\mbox{}\hfill 17 &\mbox{}\hfill 28 &\mbox{}\hspace*{\fill}\textbf{ 30}&\mbox{}\hfill 22 &\mbox{}\hspace*{\fill}\textbf{  9}&\mbox{}\hspace*{\fill}\textbf{  0}&\mbox{}\hspace*{\fill}\textbf{  0}\\
 9~~&\mbox{}\hspace*{\fill}\textbf{  1}&\mbox{}\hspace*{\fill}\textbf{  7}&\mbox{}\hfill 23 &\mbox{}\hfill 45 &\mbox{}\hfill 58 &\mbox{}\hfill 51 &\mbox{}\hspace*{\fill}\textbf{ 27}&\mbox{}\hspace*{\fill}\textbf{  9}&\mbox{}\hspace*{\fill}\textbf{  3}&\mbox{}\hspace*{\fill}\textbf{  1}\\
10~~&\mbox{}\hspace*{\fill}\textbf{  1}&\mbox{}\hspace*{\fill}\textbf{  8}&\mbox{}\hfill 30 &\mbox{}\hfill 68 &\mbox{}\hfill103 &\mbox{}\hfill108 &\mbox{}\hfill 78 &\mbox{}\hfill 40 &\mbox{}\hspace*{\fill}\textbf{ 18}&\mbox{}\hspace*{\fill}\textbf{  4}&\mbox{}\hspace*{\fill}\textbf{  0}\\
11~~&\mbox{}\hspace*{\fill}\textbf{  1}&\mbox{}\hspace*{\fill}\textbf{  9}&\mbox{}\hfill 38 &\mbox{}\hfill 98 &\mbox{}\hfill171 &\mbox{}\hfill211 &\mbox{}\hfill187 &\mbox{}\hspace*{\fill}\textbf{123}&\mbox{}\hfill 58 &\mbox{}\hspace*{\fill}\textbf{ 16}&\mbox{}\hspace*{\fill}\textbf{  0}&\mbox{}\hspace*{\fill}\textbf{  0}\\
12~~&\mbox{}\hspace*{\fill}\textbf{  1}&\mbox{}\hspace*{\fill}\textbf{ 10}&\mbox{}\hfill 47 &\mbox{}\hfill136 &\mbox{}\hfill269 &\mbox{}\hfill382 &\mbox{}\hfill399 &\mbox{}\hfill310 &\mbox{}\hfill176 &\mbox{}\hspace*{\fill}\textbf{ 64}&\mbox{}\hspace*{\fill}\textbf{ 16}&\mbox{}\hspace*{\fill}\textbf{  4}&\mbox{}\hspace*{\fill}\textbf{  1}\\
13~~&\mbox{}\hspace*{\fill}\textbf{  1}&\mbox{}\hspace*{\fill}\textbf{ 11}&\mbox{}\hfill 57 &\mbox{}\hfill183 &\mbox{}\hfill405 &\mbox{}\hfill651 &\mbox{}\hfill781 &\mbox{}\hfill708 &\mbox{}\hfill480 &\mbox{}\hfill240 &\mbox{}\hfill 90 &\mbox{}\hspace*{\fill}\textbf{ 30}&\mbox{}\hspace*{\fill}\textbf{  5}&\mbox{}\hspace*{\fill}\textbf{  0}\\
\end{tabular}
\caption{A Pascal-like triangle (\seqnum{A350110}) with entries
  $\protect\tchb{n}{k}_3$.}
\label{f:m=3t}
\end{figure}

\begin{figure}
\begin{tabular}{c|*{14}{p{1.6em}}}
$n$ $\backslash$ $k$&\mbox{}\hfill0&\mbox{}\hfill1&\mbox{}\hfill2&\mbox{}\hfill3&\mbox{}\hfill4&\mbox{}\hfill5&\mbox{}\hfill6&\mbox{}\hfill7&\mbox{}\hfill8&\mbox{}\hfill9&\mbox{}\hfill10&\mbox{}\hfill11&\mbox{}\hfill12&\mbox{}\hfill13\\\hline
 0~~&\mbox{}\hspace*{\fill}\textbf{  1}\\
 1~~&\mbox{}\hspace*{\fill}\textbf{  1}&\mbox{}\hspace*{\fill}\textbf{  0}\\
 2~~&\mbox{}\hspace*{\fill}\textbf{  1}&\mbox{}\hspace*{\fill}\textbf{  0}&\mbox{}\hspace*{\fill}\textbf{  0}\\
 3~~&\mbox{}\hspace*{\fill}\textbf{  1}&\mbox{}\hspace*{\fill}\textbf{  0}&\mbox{}\hspace*{\fill}\textbf{  0}&\mbox{}\hspace*{\fill}\textbf{  0}\\
 4~~&\mbox{}\hspace*{\fill}\textbf{  1}&\mbox{}\hspace*{\fill}\textbf{  1}&\mbox{}\hspace*{\fill}\textbf{  1}&\mbox{}\hspace*{\fill}\textbf{  1}&\mbox{}\hspace*{\fill}\textbf{  1}\\
 5~~&\mbox{}\hspace*{\fill}\textbf{  1}&\mbox{}\hspace*{\fill}\textbf{  2}&\mbox{}\hfill  3 &\mbox{}\hspace*{\fill}\textbf{  4}&\mbox{}\hspace*{\fill}\textbf{  2}&\mbox{}\hspace*{\fill}\textbf{  0}\\
 6~~&\mbox{}\hspace*{\fill}\textbf{  1}&\mbox{}\hspace*{\fill}\textbf{  3}&\mbox{}\hspace*{\fill}\textbf{  6}&\mbox{}\hfill  7 &\mbox{}\hspace*{\fill}\textbf{  4}&\mbox{}\hspace*{\fill}\textbf{  0}&\mbox{}\hspace*{\fill}\textbf{  0}\\
 7~~&\mbox{}\hspace*{\fill}\textbf{  1}&\mbox{}\hspace*{\fill}\textbf{  4}&\mbox{}\hfill  9 &\mbox{}\hfill 12 &\mbox{}\hspace*{\fill}\textbf{  8}&\mbox{}\hspace*{\fill}\textbf{  0}&\mbox{}\hspace*{\fill}\textbf{  0}&\mbox{}\hspace*{\fill}\textbf{  0}\\
 8~~&\mbox{}\hspace*{\fill}\textbf{  1}&\mbox{}\hspace*{\fill}\textbf{  5}&\mbox{}\hfill 13 &\mbox{}\hfill 20 &\mbox{}\hspace*{\fill}\textbf{ 16}&\mbox{}\hspace*{\fill}\textbf{  8}&\mbox{}\hspace*{\fill}\textbf{  4}&\mbox{}\hspace*{\fill}\textbf{  2}&\mbox{}\hspace*{\fill}\textbf{  1}\\
 9~~&\mbox{}\hspace*{\fill}\textbf{  1}&\mbox{}\hspace*{\fill}\textbf{  6}&\mbox{}\hfill 18 &\mbox{}\hfill 32 &\mbox{}\hfill 36 &\mbox{}\hfill 28 &\mbox{}\hfill 19 &\mbox{}\hspace*{\fill}\textbf{ 12}&\mbox{}\hspace*{\fill}\textbf{  3}&\mbox{}\hspace*{\fill}\textbf{  0}\\
10~~&\mbox{}\hspace*{\fill}\textbf{  1}&\mbox{}\hspace*{\fill}\textbf{  7}&\mbox{}\hfill 24 &\mbox{}\hfill 50 &\mbox{}\hfill 69 &\mbox{}\hfill 69 &\mbox{}\hspace*{\fill}\textbf{ 58}&\mbox{}\hfill 31 &\mbox{}\hspace*{\fill}\textbf{  9}&\mbox{}\hspace*{\fill}\textbf{  0}&\mbox{}\hspace*{\fill}\textbf{  0}\\
11~~&\mbox{}\hspace*{\fill}\textbf{  1}&\mbox{}\hspace*{\fill}\textbf{  8}&\mbox{}\hfill 31 &\mbox{}\hfill 74 &\mbox{}\hfill120 &\mbox{}\hfill144 &\mbox{}\hfill127 &\mbox{}\hfill 78 &\mbox{}\hspace*{\fill}\textbf{ 27}&\mbox{}\hspace*{\fill}\textbf{  0}&\mbox{}\hspace*{\fill}\textbf{  0}&\mbox{}\hspace*{\fill}\textbf{  0}\\
12~~&\mbox{}\hspace*{\fill}\textbf{  1}&\mbox{}\hspace*{\fill}\textbf{  9}&\mbox{}\hfill 39 &\mbox{}\hfill105 &\mbox{}\hfill195 &\mbox{}\hfill264 &\mbox{}\hfill265 &\mbox{}\hfill189 &\mbox{}\hspace*{\fill}\textbf{ 81}&\mbox{}\hspace*{\fill}\textbf{ 27}&\mbox{}\hspace*{\fill}\textbf{  9}&\mbox{}\hspace*{\fill}\textbf{  3}&\mbox{}\hspace*{\fill}\textbf{  1}\\
13~~&\mbox{}\hspace*{\fill}\textbf{  1}&\mbox{}\hspace*{\fill}\textbf{ 10}&\mbox{}\hfill 48 &\mbox{}\hfill144 &\mbox{}\hfill300 &\mbox{}\hfill458 &\mbox{}\hfill522 &\mbox{}\hfill432 &\mbox{}\hfill270 &\mbox{}\hfill132 &\mbox{}\hfill 58 &\mbox{}\hspace*{\fill}\textbf{ 24}&\mbox{}\hspace*{\fill}\textbf{  4}&\mbox{}\hspace*{\fill}\textbf{  0}\\
\end{tabular}
\caption{A Pascal-like triangle (\seqnum{A350111}) with entries
  $\protect\tchb{n}{k}_4$.}
\label{f:m=4t}
\end{figure}

We can also create a triangle of $\tch{n}{k}_m$ where
this denotes the number of tilings of an $n$-board that use $k$
$(1,m-1)$-fences (and $n-2k$ squares). The two triangles are related
via the following identity.

\begin{idn}\label{I:ch=chb}
For $n\ge k\ge0$, 
\[
\ch{n}{k}_m=\chb{n-k}{k}_m.
\]
\end{idn}
\begin{proof}
If a tiling contains $n-k$ tiles of which $k$ are fences, the total
length is $n-2k+2k=n$.
\end{proof}

As a consequence of Identity~\ref{I:ch=chb}, the antidiagonals of the
$\tchb{n}{k}_m$ triangle are the rows of the $\tch{n}{k}_m$ triangle.
In the rest of the paper we therefore only give identities for
one of the two families of tiling triangles and choose $\tchb{n}{k}_m$ as it
is more `compact' in the sense that its rows contain fewer trailing
zeros. Some of the identities, however, are more straightforward to
prove by considering the tiling of an $n$-board. The following
bijection (which is established in the proof of Theorem~5 in
\cite{EA21}) will be used in such proofs. Note that for convenience in
some of the proofs, we have extended the result to include $r=m$. This
is clearly valid as it is equivalent to changing $r$ to zero and
increasing $j$ by 1.

\begin{lemma}\label{L:bij}
For $j\ge0$ and $r=0,\ldots,m$, 
there is a bijection between the tilings of a
$(mj+r)$-board using $k$ $(1,m-1)$-fences and $mj+r-2k$ squares and
the tilings of an ordered $m$-tuple of $r$ $(j+1)$-boards
followed by $m-r$ $j$-boards using $k$ dominoes and $mj+r-2k$ squares.
\end{lemma}

Since a $(1,0)$-fence is just a domino,
$\tchb{n}{k}_1=\tbinom{n}{k}$ and so the corresponding triangle is
Pascal's triangle (\seqnum{A007318}), and $\tch{n}{k}_1$ are the
entries of the triangle \seqnum{A011973}.

\section{Relation of tiling triangles to polynomials}\label{s:poly}

Before examining the connection between the tiling triangles we
consider here and polynomials, we give a general result connecting
polynomials to triangles derived from arbitrary-length tilings when
the set of possible metatiles is finite. A \textit{metatile} is a
grouping of tiles that completely covers an integer number of cells
and cannot be split into smaller metatiles, and so any tiling of an
integer-length board can be expressed as a tiling using
metatiles \cite{Edw08}.  For example, when tiling with squares and
$(1,1)$-fences (the $m=2$ case), there are three types of metatile:
the square tile on its own (a \textit{free square}), a fence whose gap
is filled by a square (a \textit{filled fence}), and two interlocking
fences (a \textit{bifence}) \cite{EA21}.

We consider arbitrary-length tilings of boards that in general use at
least two types of tile, some but not all of which are regarded as
being special, and where the set of all possible metatiles is of finite
size $M$. We then define the following polynomial:
\begin{equation}\label{e:gpoly} 
p_n(x)=\delta_{n,0}+\sum_{i=1}^Mx^{s_i}p_{n-l_i}(x),\quad
p_{n<0}(x)=0,
\end{equation} 
where $s_i$ is the number of special tiles in the $i$-th metatile, and
$l_i$ is the length of the $i$-th metatile.  The following theorem
(which is a generalization of Combinatorial Theorem~12 in \cite{BQ=03})
relates the coefficients of $p_n(x)$ to the $n$-th row of the
triangle whose $(n,k)$-th entry is the number of tilings of an
$n$-board that use $k$ special tiles.

\begin{thm}\label{T:gpoly}
Let $p(n,k)=[x^k]p_n(x)$ and let $t(n,k)$ be the number of tilings of an
$n$-board that use exactly $k$ special tiles. Then $p(n,k)=t(n,k)$.
\end{thm}
\begin{proof}
Substituting $p_n(x)=\sum_{k=0}^\infty p(n,k)x^k$ into \eqref{e:gpoly} gives
\[
p_n(x)=\delta_{n,0}+\sum_{i=1}^M
\sum_{k=0}^\infty p(n-l_i,k)x^{k+s_i}
=\delta_{n,0}+\sum_{i=1}^M\sum_{k=s_i}^\infty p(n-l_i,k-s_i)x^k.
\]
Since $p(n,k)=0$ if $k<0$, the sum over $k$ can instead start from
zero. Then equating coefficients of $x^k$ gives
\[
p(n,k)=\delta_{n,0}\delta_{k,0}+\sum_{i=1}^Mp(n-l_i,k-s_i).
\]
For a tiling of an $n$-board containing $k$ special tiles, if the
final metatile has length $l$ and contains $s$ special tiles then
there are $t(n-l,k-s)$ 
ways to tile the rest of the board. Hence, summing over all possible metatiles, 
\[
t(n,k)=\delta_{n,0}\delta_{k,0}+\sum_{i=1}^Mt(n-l_i,k-s_i),
\]
where we regard there as being one way to tile a 0-board with no
special tiles. Since there are no ways to tile an $n$-board with with $k$
dominoes if $k<0$, we also have $t(n,k)=0$ if $k<0$.
\end{proof}

\begin{example}\label{E:poly}
When tiling an $n$-board with squares and $(\frac12,1)$-fences there
are $M=3$ types of metatile, namely, a square by itself, a fence with
its gap filled by a square, and three interlocking fences
\cite{Edw08}. If we regard the fence as the special tile, in this case
we have $l_1=1$, $l_2=2$, $l_3=3$, $s_1=0$, $s_2=1$, and
$s_3=3$. Hence row $n$ in the triangle whose $(n,k)$-th entry is the
number of tilings of an $n$-board that use $k$ $(\frac12,1)$-fences
(\seqnum{A157897}) gives the coefficients of one form of tribonacci
polynomial $t_n(x)$ which is given by
\[
t_n(x)=\delta_{n,0}+t_{n-1}(x)+xt_{n-2}+x^3t_{n-3},\quad t_{n<0}(x)=0.
\]
\end{example}

\begin{remark}
One can of course obtain an analogous polynomial to \eqref{e:gpoly} and
and an analogous result to Theorem~\ref{T:gpoly}
by instead considering $n$-tile tilings of a board in which case $l_i$
is then the number of tiles that the $i$-th metatile contains.
\end{remark}

We define the Fibonacci polynomials $f_n(x)$ by
\begin{equation}\label{e:f(x)}
f_n(x)=f_{n-1}(x)+xf_{n-2}(x)+\delta_{n,0}, \quad f_{n<0}(x)=0.
\end{equation}  
Note that Fibonacci polynomials are often instead defined as
$\bar{f}_n(x)=x\bar{f}_{n-1}(x)+\bar{f}_{n-2}(x)+\delta_{n,0}$ which
leads to a different family of polynomials (see, e.g., p.141 in
\cite{BQ=03}). The definition we use here gives $f_0(x)=f_1(x)=1$, $f_2(x)=1+x$,
$f_3(x)=1+2x$, $f_4(x)=1+3x+x^2$, $f_5(x)=1+4x+3x^2$, etc. 
It is clear from the definition that $\deg f_n(x)=\floor{n/2}$.
Notice also that
putting $x=1$ gives the sum of the coefficients and hence for both
definitions, the sum of the coefficients is the Fibonacci
number $f_n$. 

The following Lemma (which is analogous to Combinatorial Theorem~12
concerning $\bar{f}_n(x)$ in \cite{BQ=03}) relates tilings of an
$n$-board using squares and dominoes to the coefficients of $f_n(x)$.

\begin{lemma}\label{L:f=t}  
Let $f(n,k)=[x^k]f_n(x)$ and $t(n,k)$ be the number of tilings of an
$n$-board with squares and dominoes that use exactly $k$ dominoes. Then
$f(n,k)=t(n,k)$.
\end{lemma}
\begin{proof}
This and the corresponding polynomial \eqref{e:f(x)} is a particular
case of Theorem~\ref{T:gpoly} and \eqref{e:gpoly} where the metatiles
are the square and domino and the special tile is the domino.
\end{proof}

Note that when tiling an $n$-board with exactly $k$ dominoes there
will be $n-k$ tiles in total. Counting the ways to place the $k$
dominoes gives $f(n,k)=\tbinom{n-k}{k}$ which is \seqnum{A011973}, and
coefficients of ascending powers of $x$ in $f_n(x)$ therefore give the $n$th
antidiagonal of Pascal's triangle (\seqnum{A007318}), counting the initial 1
in the triangle as the 0th antidiagonal. The following theorem is a
generalization of this result.

\begin{thm}\label{T:poly}
For $j\ge0$, $k\ge0$, $m\ge1$, and  $r=0,\ldots,m-1$,
\begin{equation}\label{e:poly} 
\chb{mj+r-k}{k}_m=[x^k]f_j^{m-r}(x)f_{j+1}^r(x).
\end{equation} 
\end{thm}
\begin{proof}
From Identity~\ref{I:ch=chb},
$\tch{mj+r}{k}_m=\tchb{mj+r-k}{k}_m$. From Lemma~\ref{L:bij}, $\tch{mj+r}{k}_m$
equals the number of ways to tile an ordered $m$-tuple of $r$
$(j+1)$-boards followed by $m-r$ $j$-boards using $k$ dominoes (and
$mj+r-2k$ squares). The number of such tilings of the $m$-tuple of
boards is
\[
\sum_{\substack{k_1\ge0,\,k_2\ge0,\,\ldots,\,k_m\ge0,\\k_1+k_2+\cdots+k_m=k}}
\Biggl(\prod_{i=1}^r t(j+1,k_i)\Biggr)
\Biggl(\prod_{i=r+1}^m t(j,k_i)\Biggr)
\]
in which the first product is omitted when $r=0$.
The coefficient of $x^k$ in $f_{j+1}^r(x)f_j^{m-r}(x)$ is
\begin{align*}  
&[x^k]
\Biggl(\prod_{i=1}^r \sum_{k_i=0}^{\floor{(j+1)/2}}f(j+1,k_i)x^{k_i}\Biggr)
\Biggl(\prod_{i=r+1}^m \sum_{k_i=0}^{\floor{j/2}}f(j,k_i)x^{k_i}\Biggr)\\
&\qquad=
[x^k]\sum_{k_1\ge0,k_2\ge0,\ldots,k_m\ge0}\Biggl(\prod_{i=1}^r f(j+1,k_i)\Biggr)
\Biggl(\prod_{i=r+1}^m f(j,k_i)\Biggr)x^{k_1+k_2+\cdots+k_m}\\
&\qquad=
\sum_{\substack{k_1\ge0,\,k_2\ge0,\,\ldots,\,k_m\ge0,\\k_1+k_2+\cdots+k_m=k}}
\Biggl(\prod_{i=1}^r f(j+1,k_i)\Biggr)
\Biggl(\prod_{i=1}^r f(j,k_i)\Biggr).
\end{align*}
The result then follows from Lemma~\ref{L:f=t}.  
\end{proof}

Our first identity, which gives the sums of the antidiagonals, 
follows immediately from Theorem~\ref{T:poly}.
\begin{idn}\label{I:adiagsum} For $j\ge0$, $m\ge1$, $r=0,\ldots,m-1$, 
\[
\sum_{k=0}^{\floor{(mj+r)/2}}\chb{mj+r-k}{k}_m=f_j^{m-r}f_{j+1}^r.
\]
\end{idn}

\section{Tiling and restricted combinations}\label{s:comb}

We now turn to the problem of determining an expression for
$S^{(m)}(n,k)$, the number of subsets of $\Nset_n=\{1,\ldots,n\}$
of size $k$ such that the difference of any two elements of the subset
does not equal $m$. For example, $S^{(1)}(3,0)=S^{(1)}(3,2)=1$ and
$S^{(1)}(3,1)=3$ since the possible subsets of $\{1,2,3\}$ satisfying
the $m=1$ restriction are $\{\},\{1\},\{2\},\{3\}$, and $\{1,3\}$.  It
has been established that $S^{(1)}(n,k)=\tbinom{n+1-k}{k}$
\cite{Kap43}, and there is a formula for $S^{(m)}(n,k)$ in terms of
sums of products of binomial coefficients \cite{Pro83} along with one
in terms of products of powers of consecutive Fibonacci numbers for
the number of subsets of $\Nset_n$ of all sizes for a given $m$
\cite{KL91c}. Here we will show that $S^{(m)}(n,k)=\tchb{n+m-k}{k}_m$
and hence obtain the latter of these previous results via
combinatorial proof. We first establish the following bijection.

\begin{lemma}\label{L:ksub}
There is a bijection between the $k$-subsets of $\Nset_n$ such that
all pairs of elements taken from a subset do not differ by $m$, and
the tilings of an $(n+m)$-board with $k$ $(1,m-1)$-fences and $n+m-2k$
squares. 
\end{lemma}
\begin{proof}
We label the cells of the $(n+m)$-board from 1 to $n+m$. If a
$k$-subset contains element $i$ then we place a fence so that its left
post occupies cell $i$. Notice that if $i=n$ then the right post
occupies the final cell on the board. After placing fences
corresponding to each element of the subset, the rest of the board is
filled with squares of which there must be $n+m-2k$. In reverse, the
tiling of any $(n+m)$-board tiled with $k$ fences will generate a
$k$-subset where no two elements differ by $m$ since the right post of
a fence starting at cell $i$ is on cell $i+m$ which means it cannot be
occupied by the left post of another fence.
\end{proof}

\begin{cor}
$S^{(m)}(n,k)=\tchb{n+m-k}{k}_m$.
\end{cor}
\begin{proof}
From Lemma~\ref{L:ksub}, $S^{(m)}(n,k)=\tch{n+m}{k}_m$.
Identity~\ref{I:ch=chb} then gives the result.
\end{proof}

The next two corollaries follow from Theorem~\ref{T:poly} and
Identity~\ref{I:adiagsum}, respectively.

\begin{cor}
For $j\ge0$, $m\ge1$, $r=0,\ldots,m-1$, 
$S^{(m)}(mj+r,k)=[x^k]f_{j+1}^{m-r}(x)f_{j+2}^r(x)$.
\end{cor}

\begin{cor}
For $j\ge0$, $m\ge1$, $r=0,\ldots,m-1$, 
the number of subsets of $\Nset_{mj+r}$ each of which
lack pairs of elements that differ by $m$ is $f_{j+1}^{m-r}f_{j+2}^r$.
\end{cor}

\section{Metatiles when tiling with squares and fences}\label{s:meta}

The simplest metatiles are the free square ($S$), $m$ interlocking
fences with no gaps ($F^m$) which we will refer to as an
\textit{$m$-fence} (since the $m=2$ and $m=3$ cases have already been
referred to as bifences \cite{EA19,EA21} and trifences
\cite{Edw08,EA20a}, respectively), and, for $m>1$ and
$r=1,\ldots,m-1$, the \textit{filled $r$-fence} ($F^rS^{m-r}$) which
is $r$ interlocking fences with the remaining gap filled with squares.
We refer to a filled 1-fence ($FS^{m-1}$) simply as a filled fence.
Note that a 1-fence is just a domino, and that $S$ and $F^{m-1}S$ are
the only metatiles that contain a single square.

\begin{figure}[!h]
\begin{center}
\includegraphics[width=15cm]{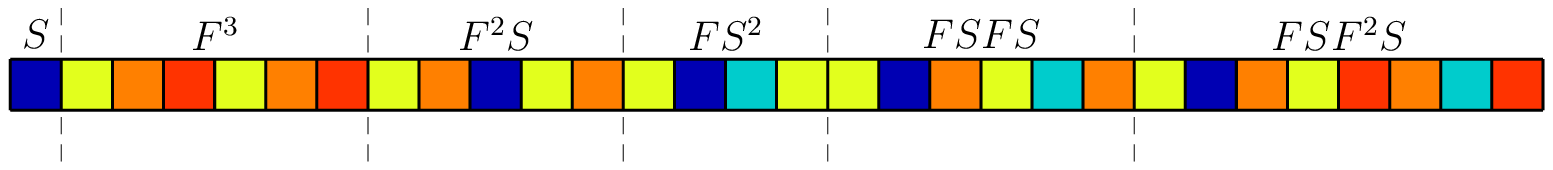}
\end{center}
\caption{A 30-board tiled with all metatiles containing less than 6 tiles
  in the $m=3$ case. Dashed lines show boundaries between
  metatiles. The symbolic representation is above each metatile.}
\label{f:meta}
\end{figure}

When $m=1$, the only metatiles are the two individual tiles
themselves: a square and a domino. When $m=2$, the metatiles are $S$,
$FS$, and $F^2$ \cite{EA21}. For $m>2$, in each case there are an
infinite number of metatiles. However, when $m=3$
(Fig.~\ref{f:meta}), aside from the simplest metatiles ($S$, $F^3$,
$F^2S$, and $FS^2$) there is just one infinite sequence of metatiles,
namely, $FSF^{j-1}S$ for $j>1$. To see this, notice that $FS$ has a
single remaining slot of unit width. This can be either filled with an
$S$, which then completes the metatile, or with an $F$ which again
results in a unit-width slot at the end of the yet-to-be-completed
metatile.

\section{Further identities concerning entries in the $n$-tile tilings
triangles}
\label{s:idn}

These first three identities follow immediately by considering
properties of the simplest metatiles.

\begin{idn}
For $n\ge0$ and $m\ge1$, $\tchb{n}{0}_m=1$.
\end{idn}
\begin{proof}
There is only one way to create an $n$-tile tiling without using any
$(1,m-1)$-fences: the all-square tiling.
\end{proof}

\begin{idn}
For $n\ge1$ and $m\ge1$, 
\[
\chb{n}{1}_m=\begin{cases}0,&n<m;\\n-m+1,&n\ge m.\end{cases}
\]
\end{idn}
\begin{proof}
Any $n$-tile tiling using exactly 1 fence must have a filled fence which
itself contains $m$ tiles. Thus there can be no $n$-tile tilings using 1
fence that use less than $m$ tiles. If $n\ge m$, the tiling consists
of a filled fence and $n-m$ free squares which gives a total of
$n-m+1$ metatile positions in which the filled fence can be placed.
\end{proof}

\begin{idn}
For $j\ge0$, $m\ge1$, and $r=0,\ldots,m-1$,  
$\tchb{n}{n}_m=\delta_{n\bmod m,0}$.
\end{idn}
\begin{proof}
The only way to tile without squares is the all $m$-fence tiling which
can only occur if the number of tiles is a multiple of $m$.
\end{proof}

The pattern of zeros seen in the triangles is a result of the
following identity.
\begin{idn}
For $j\ge1$, $m\ge1$, $p=0,\ldots,m-1$, and $r=1,\ldots,p$, 
\[
\chb{mj-r}{mj-p}_m=0.
\]
\end{idn}
\begin{proof}
We first derive an expression for $K$, the maximum number of fences
that can be used in the tiling of an $(mJ+R)$-board where
$R=0,\ldots,m-1$.  From Lemma~\ref{L:bij}, $K$ is also the maximum
number of dominoes that can be used in the tiling of $R$
$(J+1)$-boards and $m-R$ $J$-boards. Then it is easily seen that
\[
K=\begin{cases}\tfrac12mJ, & \text{$J$ even};\\
 \tfrac12m(J-1)+R, & \text{$J$ odd}.\end{cases}
\]
From Identity~\ref{I:ch=chb}, 
\[
\chb{mj-r}{mj-p}_m=\ch{2mj-r-p}{mj-p}_m.
\]
If $r+p>m$, then $2mj-r-p=mJ+R$ where $J=2(j-1)$ and $R=2m-r-p$. Then
$K=m(j-1)$ which is always less than $mj-p$. If $r+p\le m$ then
$2mj-r-p=mJ+R$ where $J=2j-1$ and $R=m-r-p$ and so $K=mj-p-r$ which is
also always less than $mj-p$.    
\end{proof}

The following identity accounts for the rising and falling
powers of ascending positive integers that form the right 
boundary of the nonzero parts of the triangles.
\begin{idn}
For $j\ge1$, $m\ge1$, and $p=0,\ldots,m$, 
\[
\chb{m(j-1)+p}{m(j-1)}_m=\chb{mj}{mj-p}_m=j^p.
\]
\end{idn}
\begin{proof}
From Identity~\ref{I:ch=chb},
\[ 
\chb{m(j-1)+p}{m(j-1)}_m=\ch{2m(j-1)+p}{m(j-1)}_m.
\]
By
Lemma~\ref{L:bij} this is the number of ways to tile $m-p$
boards of length $2(j-1)$ and $p$ boards of length $2(j-1)+1$ with $m(j-1)$
dominoes and $p$ squares. Putting $j-1$ dominoes in each of the $m$
boards leaves room for the remaining $p$ squares in the set of $p$
longer boards. On each of these boards there are $j$ tiles and hence
$j$ ways to tile each of them leading to $j^p$ ways to tile all the boards.   
From Identity~\ref{I:ch=chb} we have
\[
\chb{mj}{mj-p}_m=\ch{2mj-p}{mj-p}_m=\ch{m(2j-1)+m-p}{(m-p)j+p(j-1)}_m,
\]
which is also the number of ways to tile $m-p$ $2j$-boards and $p$
boards of length $2j-1$ using $(m-p)j+p(j-1)$ dominoes and $p$
squares. The $2j$-boards are completely filled by $j$ dominoes and the
$p$ $(2j-1)$-boards each have $j-1$ dominoes and one square which can
be placed in $j$ positions leading again to a total of $j^p$ tilings
for the set of boards.  
\end{proof}

\begin{idn}
For $j\ge1$ and $m\ge1$, 
\[
\chb{mj+1}{mj-1}_m=mT_j,
\]
where $T_j=j(j+1)/2$ is the $j$th triangle number (\seqnum{A000217}).
\end{idn}
\begin{proof}
From Identity~\ref{I:ch=chb}, $\tchb{mj+1}{mj-1}_m=\tch{2mj}{mj-1}$,
which, from Lemma~\ref{L:bij}, is the number of ways to tile an
$m$-tuple of $2j$-boards with $mj-1$ dominoes and 2 squares. As the
boards are of even length, both squares must lie on the same
board. On such a board there are $j+1$ tiles in total which means
there are $\tbinom{j+1}{2}=j(j+1)/2$ possible ways to tile it. 
As there are $m$ possible boards on which to place the two squares, the
result follows.
\end{proof}

\begin{idn}
For $j\ge1$ and $m\ge2$, 
\[
\chb{mj+2}{mj-2}_m=m\binom{j+2}{4}1_{j>1}+\binom{m}{2}\binom{j+1}{2}^2,
\]
where $1_x$ is $1$ if $x$ is true and $0$ otherwise.
\end{idn}
\begin{proof}
From Identity~\ref{I:ch=chb}, $\tchb{mj+2}{mj-2}_m=\tch{2mj}{mj-2}$,
which, from Lemma~\ref{L:bij}, is the number of ways to tile an
$m$-tuple of $2j$-boards with $mj-2$ dominoes and 4 squares.  If
$j>1$, all four squares can be on the same $2j$-board which means
there are $j+2$ tiles on that board and hence $\tbinom{j+2}{4}$
tilings of it. With $m$ boards to choose from, this gives the first term on
the right-hand side of the identity. The other possibility is that two
of the boards have two squares each. There are $\tbinom{j+1}{2}$ ways
to tile each such board and $\tbinom{m}{2}$ ways to choose the boards.
\end{proof}

We find that $\tchb{2n+2}{2n-2}_2$ is \seqnum{A006324}.

As the metatiles containing a given number of tiles are easily
enumerated, it is straightforward to obtain recursion relations for
the row sums and for the elements of the triangle in the $m=3$ case,
as shown in the proofs of the following two identities.

\begin{idn}
For all $n\in\mathbb{Z}$, $B_n=\sum_{k=0}^n\tchb{n}{k}_3$, the number of $n$-tile tilings using
squares and $(1,2)$-fences, satisfies
\begin{equation}\label{e:B3} 
B_n=\delta_{n,0}-\delta_{n,1}+2B_{n-1}-B_{n-2}+3B_{n-3}-2B_{n-4},
\quad B_{n<0}=0.
\end{equation} 
\end{idn}
\begin{proof}
We condition on the final metatile in the $n$-tile tiling. If it contains $p$ tiles then there are $B_{n-p}$ possibilities for the remaining tiles. The possible metatiles, namely, $S$, $F^3$, $F^2S$, $FS^2$, and $FSF^{j-1}S$ for $j>1$ contain 1, 3, 3, 3, and $2+j$ tiles, respectively. Summing over these gives
\begin{equation}\label{e:rs3} 
B_n=\delta_{n,0}+B_{n-1}+3B_{n-3}+\sum_{p=4}^nB_{n-p},
\end{equation}
where the $\delta_{n,0}$ is needed so that we obtain one tiling for each metatile with $p$ tiles (putting $n=p$). Subtracting \eqref{e:rs3} with $n$ replaced by $n-1$ from \eqref{e:rs3} gives the identity.
\end{proof}

\begin{idn}\label{I:rr3}
For all $n,k\in\mathbb{Z}$,
\begin{multline}\label{e:rr3}
\chb{n}{k}_3=\delta_{n,0}\delta_{k,0}-\delta_{n,1}\delta_{k,1}
+\chb{n-1}{k}_3+\chb{n-1}{k-1}_3
-\chb{n-2}{k-1}_3+\chb{n-3}{k-1}_3+\chb{n-3}{k-2}_3\\
+\chb{n-3}{k-3}_3-\chb{n-4}{k-3}_3-\chb{n-4}{k-4}_3.
\end{multline} 
\end{idn}
\begin{proof}
We count $\tchb{n}{k}$ by conditioning on the last metatile on the
board. If the metatile contains $p$ tiles of which $j$ are fences, for
the remaining tiles the number of $(n-p)$-tile tilings is
$\tchb{n-p}{k-j}$. 
Summing over all possible metatiles gives
\begin{equation}\label{e:irr3}
\chb{n}{k}_3=\delta_{n,0}\delta_{k,0}+\chb{n-1}{k}_3+\chb{n-3}{k-3}_3+
\chb{n-3}{k-1}_3+\chb{n-3}{k-2}_3+\sum_{p=4}^n\chb{n-p}{k+2-p}_3.
\end{equation} 
Replacing $n$ by $n-1$ and $k$ by $k-1$ in \eqref{e:irr3} and then
subtracting the resulting equation from \eqref{e:irr3} gives the identity.
\end{proof}

\begin{cor}
For $n\geq2k+1$ when $k\geq0$, 
\[
\chb{n}{k}_3=\chb{n-1}{k}_3+\chb{n-1}{k-1}_3.
\]
\end{cor}
\begin{proof}
We define what we might call a Pascal's recurrence operator by
\[
P(n,k)=\chb{n}{k}_3-\chb{n-1}{k}_3-\chb{n-1}{k-1}_3.
\]
Notice that $P(n,k)=0$ if the $(n,k)$-th entry of the triangle is the sum
of the entry directly above it and the entry above and one place
to the left.  It can be seen that $P(n,k<0)=0$ and $P(n>0,0)=0$. 
Rewriting \eqref{e:rr3} in terms of $P(n,k)$ gives
\[
P(n,k)=\delta_{n,0}\delta_{k,0}-\delta_{n,1}\delta_{k,1}-P(n-2,k-1)+P(n-3,k-3),
\]
and hence for $k>0$ and $r\ge0$,
\[
P(2k+1+r,k)=-P(2(k-1)+1+r,k-1)+P(2k-2+r,k-3).
\]
Applying this recursively until the second argument of $P$ is zero or
negative in each case, we see that the first term on the right-hand
side will eventually generate $(-1)^kP(1+r,0)$ and all other terms
will be of the form $\sigma P(a,0)$ for various $a>1+r$ or $\sigma
P(a,b)$ with $b\in\{-1,-2\}$ where $\sigma$ is 1 or $-1$. Hence
$P(2k+1+r,k)=0$ for $k\ge0$ which is equivalent to the result we wish
to prove.
\end{proof}

The following conjecture has been shown to be true for the cases $m=1,2,3$.
\begin{conj}
For $n\geq(m-1)k+1$ when $k\ge0$, 
\[
\chb{n}{k}_m=\chb{n-1}{k}_m+\chb{n-1}{k-1}_m.
\]
\end{conj}

\section{Tiling triangles and Riordan arrays}\label{s:rio}

It is well known that an array is a Riordan array if and only if there
is an $A$-sequence as given by \eqref{e:Aseq} \cite{Rog78}. The
problem with such a recursion relation in tiling triangle applications
is that the expression for a given element is in terms of elements
only in the row above whereas the most readily obtained recursion
relations for tiling triangles tend to be higher than first order in
the row number $n$. In such cases, the following more general
characterization of Riordan arrays (using an `$A$-matrix' rather than
an $A$-sequence) is needed to show whether or not a triangle (or its
row-reversed version) is a Riordan array.

\begin{thm}[Theorem 2.5 of \cite{MRSV97}]\label{T:Amat}
A matrix $\{R_{n,k}\}_{n,k\in\Nset}$, where $\Nset=\{0,1,\ldots\}$, is
a Riordan array if and only if there is another matrix
$\{A_{i,j}\}_{i,j\in\Nset}$ with $A_{0,0}\neq0$ such that every
$R_{n,k}$ for $n,k\ge1$ can be expressed as
$R_{n,k}=\sum_{i\ge1}\sum_{j\ge-1}A_{i,j}R_{n-i,k+j}$.
\end{thm}

Notice that the above theorem means that an element in a Riordan array 
can depend on any elements above it as long as they are
not more than 1 column before it.

\begin{remark}\label{R:a>=b}
 As the recursion relation for
elements $T(n,k)$ of tiling triangles are constructed
by conditioning on the final metatile (as in the proof of
Identity~\ref{I:rr3}), such recursion relations will only involve
terms $T(n-a,k-b)$ where $a,b\in\Nset$. Furthermore, the tiling
triangle will only be a lower triangular matrix (and hence it or its
row-reversed version a candidate for a Riordan array) if $a\geq
b$. Recursion relations with terms having $b>a$ can occur if $k$
counts tiles whose sub-tiles have a total length of less than 1.  
\end{remark}

\begin{thm}\label{T:rio}
Suppose a triangle is constructed by letting the $(n,k)$-th entry be
the number of ways to tile an $n$-board that use $k$ special
tiles. The triangle is a Riordan array if and only if there is a
metatile of length 1 that contains exactly one special tile and there
is no metatile that contains more than one special tile.  The triangle
is a row-reversed Riordan array if and only if there is a metatile of
length 1 that lacks special tiles and for all metatiles $l-s$ is 0 or 1,
where $l$ is the length of the metatile and $s$ is the number of
special tiles it contains.
\end{thm}
\begin{proof}
If $T(n,k)$ is the $(n,k)$-th entry, then conditioning on the final
metatile gives
\begin{equation}\label{e:gentri} 
T(n,k)=\delta_{n,0}\delta_{k,0}+\sum_iT(n-l_i,k-s_i),
\end{equation} 
where $l_i$ is the length of the $i$-th metatile and $s_i$ is the
number of special tiles it contains.  In order that $T(n,k)$ meets the
condition to be a Riordan array, we require at least one $i$ with
$l_i=s_i=1$ (so that $A_{0,0}\neq0$ in Theorem~\ref{T:Amat})
and $s_i$ cannot exceed 1. Let $\bar{T}(n,k)$ be the $(n,k)$-th entry
in the row-reversed triangle (reversing entries in each row up to and
including the main diagonal). Replacing $T(a,b)$ by
$\bar{T}(a,a-b)$ in \eqref{e:gentri} and then replacing $n-k$ by $k$
leaves
\begin{equation}\label{e:rrgentri} 
\bar{T}(n,k)=\delta_{n,0}\delta_{k,0}+\sum_i\bar{T}(n-l_i,k-(l_i-s_i)). 
\end{equation} 
The reversed triangle is a Riordan array if and only if there is an
$i$ such that $l_i=1$ and $s_i=0$ (i.e., there is a metatile of length
1 that lacks special tiles) and if $l_i-s_i$ is 0 or 1 for all $i$
(since $l_i-s_i$ cannot exceed 1 by Theorem~\ref{T:Amat} and cannot be
negative as discussed in Remark~\ref{R:a>=b}). 
\end{proof}

\begin{example}
When tiling an $n$-board with half-squares (i.e., $\frac12\times1$ tiles always
placed with the shorter sides horizontal) and
$(\frac12,\frac12)$-fences, all metatiles contain not more than 2
half-squares \cite{EA19}. If the fence is regarded as the special
tile, the associated triangle (\seqnum{A123521}) is a row-reversed
Riordan array since two half-squares make a metatile of length 1 and
since, for this type of tiling, $l_i-s_i$ equates to the total length
of the half-squares in the $i$-th metatile, it is either 0 or 1. The
triangle is in fact the $(1/(1-x^2),x/(1-x)^2)$ row-reversed Riordan
array \cite{EA20}.
\end{example}

\begin{thm}\label{T:rioT}
Suppose a triangle is constructed by letting the $(n,k)$-th entry be
the number of $n$-tile tilings of a board that use $k$ special
tiles. The triangle is a Riordan array if and only if the
special tiles are metatiles consisting of just one tile and there is no
metatile that contains more than one special tile. The triangle is a
row-reversed Riordan array if and only if there is a metatile
consisting of a single tile which is not a special tile and
no metatile contains more than one non-special tile.
\end{thm}
\begin{proof}
If $T(n,k)$ is the $(n,k)$-th entry, then conditioning on the final
metatile gives
\begin{equation}\label{e:gentriT} 
T(n,k)=\delta_{n,0}\delta_{k,0}+\sum_iT(n-p_i,k-s_i),
\end{equation} 
where $p_i$ is the number of tiles contained in the $i$-th metatile and
$s_i$ is the number of special tiles it contains.  For the triangle
to be a Riordan array, we
require at least one $i$ with $p_i=s_i=1$ (which means the special
tiles are themselves metatiles) and $s_i$ cannot exceed 1 (which means
no metatile can contain more than one special tile). Let
$\bar{T}(n,k)$ be the $(n,k)$-th entry in the row-reversed triangle.
Then
\begin{equation}\label{e:rrgentriT} 
\bar{T}(n,k)=\delta_{n,0}\delta_{k,0}+\sum_i\bar{T}(n-p_i,k-(p_i-s_i)). 
\end{equation} 
The reversed triangle is a Riordan array if and only if there is an
$i$ such that $p_i=1$ and $s_i=0$ (i.e., there is a metatile composed
of a single non-special tile) and if $p_i-s_i$ is 1 or 0 for all $i$
(for reasons given in the proof of Theorem~\ref{T:rio}). 
\end{proof}

In some instances, every other row of a tiling triangle or its
row-reversed version are Riordan arrays \cite{EA21}. Analogous theorems can
be applied to test for this.

\begin{remark}
If a tiling triangle or its row-reversed version is a $(p,q)$ Riordan
array, then determining the functions $p$ and $q$ is
straightforward. If it turns out that the tiling triangle is a Riordan
array (row-reversed Riordan array) then $p$ is the generating function
for tilings with no (with only) special tiles. To find $q(x)$, each term
$T(n-a,k-b)$ (or $\bar{T}(n-a,k-b)$) in the recursion relation 
is replaced by $x^apq^b$ \cite{EA21}. As $b$ is either 0 or 1, 
dividing by $pq^{k-1}$ gives an equation which is linear in $q$.
\end{remark}

We end by giving a corollary of Theorem~\ref{T:rioT} applied to the
$n$-tile tiling family of triangles. 
\begin{cor}
The $\tchb{n}{k}_m$ triangles for $m>2$ are not
row-reversed Riordan arrays.
\end{cor}
\begin{proof}
When $m>2$ there is at least one metatile containing 2 squares
(namely, the filled fence $FS^{m-1}$).
\end{proof}

\section{Discussion}

Further generalizations of the tiling triangles presented here are
possible if one tiles with squares and combs (which are
generalizations of fences that can have more than two sub-tiles
\cite{AE-GenFibSqr}). This will be the subject of a future article. 

Just as tiling an $n$-board with $(\frac12,g)$-fences turned out to be
a natural way to envisage permutations of $\{1,2,\ldots,n\}$
\cite{EA15}, we have shown here that a natural representation of the
subsets of $\{1,2,\ldots,n\}$ where no two elements differ by $m$ is
the tilings of an $(n+m)$-board using squares and $(1,m-1)$-fences.
This representation can be extended to the case where there are
multiple disallowed differences of the subset elements by tiling with
squares and combs, but we will address this elsewhere.

We have given conditions for a tiling triangle or its row-reversed
version to be a Riordan array. An open question is whether it is
possible to find a tiling interpretation of an arbitrary (possibly
row-reversed) Riordan array. We have been unable to do so for the
first family of triangles we considered here for $m>2$.


\bigskip
\hrule
\bigskip

\noindent 2010 {\it Mathematics Subject Classification}:
Primary 11B39; Secondary 05A19, 05A15.

\noindent \emph{Keywords}:
combinatorial proof, combinatorial identity, $n$-tiling,
Pascal-like triangle, Riordan array, Fibonacci polynomial,
restricted combination

\bigskip
\hrule
\bigskip

\sloppy
\noindent (Concerned with sequences 
\seqnum{A000045}, 
\seqnum{A000079}, 
\seqnum{A000217}, 
\seqnum{A006324}, 
\seqnum{A006498}, 
\seqnum{A007318}, 
\seqnum{A011973}, 
\seqnum{A059259}, 
\seqnum{A077947}, 
\seqnum{A079962}, 
\seqnum{A115451}, 
\seqnum{A118923}, 
\seqnum{A123521}, 
\seqnum{A157897}, 
\seqnum{A335964}, 
\seqnum{A349839}, 
\seqnum{A349840}, 
\seqnum{A349841},
\seqnum{A349842},
\seqnum{A349843},
\seqnum{A350110}, 
\seqnum{A350111}, 
and \seqnum{A350112}.) 

\bigskip
\hrule
\bigskip

\end{document}